\documentclass[reqno]{amsart}
\usepackage{amsmath}

\usepackage{amsfonts}
\usepackage{amssymb}
\usepackage{hyperref}

\begin{document}
\title[ Positive solutions ]{Positive solutions of a nonlinear fourth-order integral boundary value problem}
\author[S. Benaicha, F. Haddouchi]{Slimane Benaicha, Faouzi Haddouchi}
\address{Slimane Benaicha \\
Laboratory of Fundamental and Applied Mathematics of Oran (LMFAO) \\
Department of Mathematics, University of Oran 1 Ahmed Benbella, 31000 Oran,
Algeria} \email{slimanebenaicha@yahoo.fr}
\address{Faouzi Haddouchi\\
Department of Physics, University of Sciences and Technology of
Oran-MB, El Mnaouar, BP 1505, 31000 Oran, Algeria \\
Laboratory of Fundamental and Applied Mathematics of Oran (LMFAO)\\
Department of Mathematics, University of Oran 1 Ahmed Benbella, 31000 Oran,
Algeria}
\email{fhaddouchi@gmail.com}

\begin{abstract}
In this paper, the existence of positive solutions for a nonlinear fourth-order two-point boundary value problem with integral condition is investigated. By using Krasnoselskii's fixed point theorem on cones, sufficient conditions for the existence of at least one positive solutions are obtained.
\end{abstract}

\maketitle \numberwithin{equation}{section}
\newtheorem{theorem}{Theorem}[section]
\newtheorem{lemma}[theorem]{Lemma}
\newtheorem{definition}[theorem]{Defnition}
\newtheorem{proposition}[theorem]{Proposition}
\newtheorem{corollary}[theorem]{Corollary}
\newtheorem{remark}[theorem]{Remark}
\newtheorem{exmp}{Example}[section]

\section{Introduction}
In this work, we study the existence of positive solutions of a nonlinear two-point boundary value problem (BVP) for the following fourth-order differential equation:

\begin{equation} \label{eq-1.1}
{u^{\prime \prime \prime \prime }}(t)+f(u(t))=0,\ t\in(0,1),
\end{equation}
\begin{equation} \label{eq-1.2}
u^{\prime}(0)=u^{\prime}(1)=u^{\prime \prime}(0)=0, \ u(0)=\int_{0}^{1}a(s)u(s)ds,
\end{equation}
where
\begin{itemize}
\item[(H1)] $f\in C([0,\infty),[0,\infty))$;
\item[(H2)] $a\in C([0,1],[0,\infty))$ and $0<\int_{0}^{1}a(s)ds<1$.
\end{itemize}

Fourth-order ordinary differential equations are models for bending or deformation of elastic beams, and therefore have important applications in engineering and physical sciences. Recently, the two-point and multi-point boundary value problems for fourth-order nonlinear differential equations have received  much attention from many authors. Many authors have studied the beam equation under various boundary conditions and by different approaches.

In 2009, Graef et al.\cite{Graef3} considered the fourth order three-point boundary value problem
\begin{equation} \label{eq-1.3}
{u^{\prime \prime \prime \prime }}(t)=g(t)f(u(t)), \ t\in(0,1),
\end{equation}
together with the boundary conditions
\begin{equation} \label{eq-1.4}
u(0)=u^{\prime}(0)=u^{\prime \prime}(\beta)=u^{\prime \prime}(1)=0.
\end{equation}

In 2006, Anderson and Avery \cite{Ander}, studied the following fourth order right focal four-point boundary value problem
\begin{equation} \label{eq-1.5}
{u^{\prime \prime \prime \prime }}(t)+f(u(t))=0, \ 0<t<1,
\end{equation}
\begin{equation} \label{eq-1.6}
u(0)=u^{\prime}(q)=u^{\prime \prime}(r)=u^{\prime \prime \prime}(1)=0.
\end{equation}

In 2011, Xiading Han et al.\cite{Han}, considered the nonlocal fourth-order boundary value problem with variable parameter
\begin{equation} \label{eq-1.7}
{u^{\prime \prime \prime \prime }}(t)+B(t)u^{\prime \prime}(t)=\lambda f(t,u(t), u^{\prime \prime}(t)), \ 0<t<1,
\end{equation}
\begin{equation} \label{eq-1.8}
u(0)=u(1)=\int_{0}^{1}p(s)u(s)ds, \ u^{\prime \prime}(0)=u^{\prime \prime}(1)=\int_{0}^{1}q(s)u^{\prime \prime}(s)ds.
\end{equation}

In 2013, Yan Sun and Cun Zhu \cite{Sun}, considered the singular fourth-order three-point boundary value problem
\begin{equation} \label{eq-1.9}
{u^{\prime \prime \prime \prime }}(t)+f(t,u(t))=0, \ 0<t<1,
\end{equation}
\begin{equation} \label{eq-1.10}
u(0)=u^{\prime}(0)=u^{\prime \prime}(0)=0 , \ u^{\prime \prime}(1)-\alpha u^{\prime \prime}(\eta)=\lambda.
\end{equation}

In 2014, Xiaorui Liu and Dexiang Ma \cite{Liu}, considered the third-order two-point boundary value problem
\begin{equation} \label{eq-1.11}
{u^{\prime \prime \prime }}(t)+f(u(t))=0, \ 0<t<1,
\end{equation}
\begin{equation} \label{eq-1.12}
u^{\prime}(0)=u^{\prime}(1)=0 , \ u(0)=\int_{0}^{1}k(s)u(s)ds,
\end{equation}
and in 2015, Wenguo shen \cite{Shen}, studied the fourth-order second-point nonhomogeneous singular boundary value problem
\begin{equation} \label{eq-1.9}
{u^{\prime \prime \prime \prime }}(t)+a(t)f(u(t))=0, \ 0<t<1,
\end{equation}
\begin{equation} \label{eq-1.10}
u(0)=\alpha, \ u(1)=\beta, \ u^{\prime}(0)=\lambda, \ u^{\prime}(1)=-\mu.
\end{equation}

For some other results on boundary value problem, we refer the reader to the papers
\cite{Alves, Graef1, Graef2, Hend, Kos, Ma1, Ma2, Ping, Webb, Yao, Yang, Zhang}.

To the authors’ knowledge, there are few papers that have considered the existence of solutions for fourth-order two-point boundary value problem with integral condition. Motivated by the works mentioned above, the aim of this paper is to establish some sufficient conditions for the existence of at least one positive solutions of the BVP \eqref{eq-1.1} and \eqref{eq-1.2}.

We shall first construct the Green's function for the associated linear boundary value problem and then determine the properties of the Green's function for associated linear boundary value problem. Finally, existence results for at least one positive solution for the above problem are established when $f$ is superlinear or sublinear.
As applications, some interesting examples are presented to illustrate the main results.

\section{Preliminaries}
We shall consider the Banach space $C([0,1])$ equipped with the sup norm
\[\|u\|=\\sup_{t\in[0, 1]}|u(t)|\]

\begin{definition}\label{def 2.1}
Let $E$ be a real Banach space. A nonempty, closed, convex set $
K\subset E$ is a cone if it satisfies the following two conditions:

(i) $x\in K$, $\lambda \geq 0$ imply $\lambda x\in K$;

(ii) $x\in K$, $-x\in K$ imply $x=0$.
\end{definition}

\begin{definition}\label{def 2.2}
An operator $T:E\rightarrow E$ \ is completely continuous if it is continuous
and maps bounded sets into relatively compact sets.
\end{definition}

To prove some of our results, we will use the following fixed point theorem, which is due to Krasnoselskii's \cite{Krasn}.

\begin{theorem}  \label{thm 2.3}\cite{Krasn}.
Let $E$ be a Banach space, and let $K\subset E$, be a cone. Assume that $%
\Omega_{1}$ and $\Omega_{2}$ are open subsets of $E$ with $0\in \Omega _{1}$,
$\Omega _{1}\subset \Omega_{2}$ and let
\[
A:K\cap (\overline{
\Omega_{2}}\backslash \Omega_{1})\rightarrow K
\]
be a completely continuous operator such that

(i) $\ \left\Vert Au\right\Vert \leq \left\Vert u\right\Vert ,$ $u\in K\cap
\partial
\Omega _{1}$, and $\left\Vert Au\right\Vert \geq \left\Vert u\right\Vert ,$
$u\in K\cap \partial \Omega_{2}$; or

(ii) $\left\Vert Au\right\Vert \geq \left\Vert u\right\Vert ,$ $u\in K\cap
\partial
\Omega_{1}$, and \ $\left\Vert Au\right\Vert \leq \left\Vert u\right\Vert ,$
$u\in K\cap \partial \Omega_{2}.$

Then $A$ has a fixed point in $K\cap (\overline{\Omega _{2}}$ $\backslash $ $
\Omega_{1})$.
\end{theorem}

Consider the two-point boundary value problem

\begin{equation} \label{eq-2.1}
{u^{\prime \prime \prime \prime }}(t)+y(t)=0,\ t\in(0,1),
\end{equation}
\begin{equation} \label{eq-2.2}
u^{\prime}(0)=u^{\prime}(1)=u^{\prime \prime}(0)=0, \ u(0)=\int_{0}^{1}a(s)u(s)ds.
\end{equation}

\begin{lemma}\label{lem 2.4}
The problem \eqref{eq-2.1}-\eqref{eq-2.2} has a unique solution
\[ u(t)=\int_{0}^{1}\Big(G(t, s)+\frac{1}{1-\alpha}\int_{0}^{1}a(\tau)G(\tau, s)d\tau \Big)y(s)ds,\]
where $G(t, s):[0, 1]\times[0, 1]\rightarrow \mathbb{R}$ is the Green's function defined by
\begin{equation} \label{eq-2.3}
G(t, s)=\frac{1}{6}\begin{cases} t^{3}(1-s)^{2}-(t-s)^{3},&0\leq s\leq t \leq1;  \\
t^{3}(1-s)^{2}, &0\leq t \leq s\leq 1,
\end{cases}
\end{equation}
and \[ \alpha=\int_{0}^{1}a(t)dt.\]
\end{lemma}

\begin{proof}
Integrating \eqref{eq-2.1} over the interval $[0, t]$ for $t\in[0, 1]$, we obtain
\begin{align*}
&u^{\prime \prime \prime}(t)=-\int_{0}^{t}y(s)ds+C_{1},\\
&u^{\prime \prime}(t)=-\int_{0}^{t}(t-s)y(s)ds+C_{1}t+C_{2},\\
&u^{\prime}(t)=-\frac{1}{2}\int_{0}^{t}(t-s)^{2}y(s)ds+\frac{1}{2}C_{1}t^{2}+C_{2}t+C_{3},\\
\end{align*}
\begin{align}\label{eq-2.4}
&u(t)=-\frac{1}{6}\int_{0}^{t}(t-s)^{3}y(s)ds+\frac{1}{6}C_{1}t^{3}+\frac{1}{2}C_{2}t^{2}+C_{3}t+C_{4}.
\end{align}
From the boundary conditions \eqref{eq-2.2} we get
\[C_{2}=C_{3}=0, \ \ C_{1}=\int_{0}^{1}(1-s)^{2}y(s)ds,\]
and
\begin{eqnarray*}
C_{4}&=&u(0)\\
&=&\int_{0}^{1}a(\tau)\Big(-\frac{1}{6}\int_{0}^{\tau}(\tau-s)^{3}y(s)ds+
\frac{1}{6}\tau^{3}\int_{0}^{1}(1-s)^{2}y(s)ds+C_{4}\Big)d\tau \\
&=&-\frac{1}{6}\int_{0}^{1}a(\tau)\Big(\int_{0}^{\tau}(\tau-s)^{3}y(s)ds\Big)d\tau+
\frac{1}{6}\int_{0}^{1}a(\tau)\tau^{3}\Big(\int_{0}^{1}(1-s)^{2}y(s)ds\Big)d\tau\\
&&+ C_{4}\int_{0}^{1}a(\tau)d\tau,
\end{eqnarray*}
which implies
\[C_{4}=\frac{1}{6(1-\alpha)}\Big(\int_{0}^{1}a(\tau)\tau^{3}\Big(\int_{0}^{1}(1-s)^{2}y(s)ds\Big)d\tau-
\int_{0}^{1}a(\tau)\Big(\int_{0}^{\tau}(\tau-s)^{3}y(s)ds\Big)d\tau\Big).\]
Replacing these expressions in \eqref{eq-2.4}, we get
\begin{eqnarray*}
u(t)&=&-\frac{1}{6}\int_{0}^{t}(t-s)^{3}y(s)ds+\frac{1}{6}t^{3}\int_{0}^{1}(1-s)^{2}y(s)ds\\
&&+\frac{1}{6(1-\alpha)}\Big[\int_{0}^{1}a(\tau)\tau^{3}\Big(\int_{0}^{1}(1-s)^{2}y(s)ds\Big)d\tau\\
&&-\int_{0}^{1}a(\tau)\Big(\int_{0}^{\tau}(\tau-s)^{3}y(s)ds\Big)d\tau\Big]\\
&=&\frac{1}{6}\int_{0}^{t}\Big[t^{3}(1-s)^{2}-(t-s)^{3}\Big]y(s)ds+\frac{1}{6}\int_{t}^{1}t^{3}(1-s)^{2}y(s)ds\\
&&+\frac{1}{6(1-\alpha)}\Big[\int_{0}^{1}a(\tau)\Big(\int_{0}^{\tau}\Big[\tau^{3}(1-s)^{2}-(\tau-s)^{3}\Big]y(s)ds\\
&&+\int_{\tau}^{1}\tau^{3}(1-s)^{2}y(s)ds\Big)d\tau\Big]\\
&=&\frac{1}{6}\int_{0}^{t}\Big[t^{3}(1-s)^{2}-(t-s)^{3}\Big]y(s)ds+\frac{1}{6}\int_{t}^{1}t^{3}(1-s)^{2}y(s)ds\\
&&+\frac{1}{1-\alpha}\int_{0}^{1}\Big(\frac{1}{6}\int_{0}^{\tau}a(\tau)\Big[\tau^{3}(1-s)^{2}-(\tau-s)^{3}\Big]y(s)ds\\
&&+\frac{1}{6}\int_{\tau}^{1}a(\tau)\tau^{3}(1-s)^{2}y(s)ds\Big)d\tau\\
&=&\int_{0}^{1}G(t, s)y(s)ds+\frac{1}{1-\alpha}\int_{0}^{1}\Big(\int_{0}^{1}a(\tau)G(\tau, s)y(s)ds\Big)d\tau\\
&=&\int_{0}^{1}\Big(G(t, s)+\frac{1}{1-\alpha}\int_{0}^{1}a(\tau)G(\tau, s)d\tau\Big)y(s)ds.
\end{eqnarray*}
\end{proof}

\begin{lemma}\label{lem 2.5}
Let $\theta\in]0, \frac{1}{2}[$ be fixed. Then
\begin{itemize}
  \item [(i)] $G(t, s)\geq0$,  \ for all $t, s\in[0, 1];$
  \item [(ii)]$\frac{1}{6}\theta^{3}s(1-s)^{2}\leq G(t, s)\leq \frac{1}{6}s(1-s)^{2}$, \ for all $(t, s)\in[\theta, 1-\theta]\times[0, 1]$.
\end{itemize}
\end{lemma}

\begin{proof}
{\rm (i)} We will show that $G(t, s)\geq0$, \ for all $(t, s)\in[0, 1]\times[0, 1]$.
Since it is obvious for $t\leq s$, we only need to prove the case $s\leq t$. Now we suppose that $s\leq t$. Then
\begin{gather}
\begin{aligned}
G(t, s)&=\frac{1}{6}\Big[t^{3}(1-s)^{2}-(t-s)^{3}\Big]=\frac{1}{6}\Big[t(t-ts)^{2}-(t-s)^{3}\Big]\\
&\geq\frac{1}{6}\Big[t(t-s)^{2}-(t-s)^{3}\Big]\\
&\geq\frac{1}{6}(t-s)^{2}\Big[t-(t-s)\Big]\\
&=\frac{1}{6}s(t-s)^{2}\geq 0.
\end{aligned}\label{eq-2.5}
\end{gather}

{\rm (ii)} If $s\leq t$, from \eqref{eq-2.3} we have
\begin{gather}
\begin{aligned}
G(t, s)&=\frac{1}{6}\Big[t^{3}(1-s)^{2}-(t-s)^{3}\Big]\geq\frac{1}{6}\Big[t^{3}(1-s)^{3}-(t-s)^{3}\Big]\\
&=\frac{1}{6}\Big[(t-ts)^{3}-(t-s)^{3}\Big]\\
&=\frac{1}{6}s(1-t)\Big[t^{2}(1-s)^{2}+t(1-s)(t-s)+(t-s)^{2}\Big]\\
&\geq\frac{1}{6}t^{2}(1-t)s(1-s)^{2}.
\end{aligned}\label{eq-2.6}
\end{gather}

On the other hand
\begin{gather}
\begin{aligned}
G(t,s)-\frac{1}{6}s(1-s)^{2}&=\frac{1}{6}t^{3}(1-s)^{2}-\frac{1}{6}(t-s)^{3}-
\frac{1}{6}s(1-s)^{2}\\
&=\frac{1}{6}s\Big(-2t^{3}+t^{3}s+3t^{2}-3ts-1+2s\Big)\\
&=\frac{1}{6}s(t-1)^{2}\Big[(s-2)t+2s-1\Big]\\
&\leq\frac{1}{6}s(t-1)^{2}\Big[(s-2)t+2t-1\Big]\\
&=\frac{1}{6}s(t-1)^{2}(st-1)\leq0.\\
\end{aligned}\label{eq-2.7}
\end{gather}
If $t\leq s$, from \eqref{eq-2.3}, we have
\begin{equation}\label{eq-2.8}
G(t, s)=\frac{1}{6}t^{3}(1-s)^{2}\geq \frac{1}{6}t^{3}s(1-s)^{2},
\end{equation}
and,
\begin{equation}\label{eq-2.9}
G(t, s)=\frac{1}{6}t^{3}(1-s)^{2}\leq\frac{1}{6}s^{3}(1-s)^{2}\leq \frac{1}{6}s(1-s)^{2}.
\end{equation}
Let
\begin{equation*}
\rho(t)=\frac{1}{6}\min\{t^{3}, t^{2}(1-t)\}=\frac{1}{6}\begin{cases} t^{3},& t\leq \frac{1}{2};  \\
t^{2}(1-t), & t \geq \frac{1}{2}.
\end{cases}
\end{equation*}
From \eqref{eq-2.6}, \eqref{eq-2.7}, \eqref{eq-2.8} and \eqref{eq-2.9} we have
\[\rho(t)s(1-s)^{2}\leq G(t, s)\leq \frac{1}{6}s(1-s)^{2}, \ \text{for all} \ (t, s)\in[0, 1]\times[0, 1].\]
For $\theta\in]0, \frac{1}{2}[$, we have
\[\frac{\theta^{3}}{6}s(1-s)^{2}\leq G(t, s)\leq \frac{1}{6}s(1-s)^{2}, \ \text{for all} \ (t, s)\in[\theta, 1-\theta]\times[0, 1].\]
\end{proof}

\begin{lemma}\label{lem 2.6}
Let $y(t)\in C([0, 1], [0, \infty))$ and $\theta\in]0, \frac{1}{2}[$. The unique solution of \eqref{eq-2.1}-\eqref{eq-2.2} is nonnegative and satisfies
\[ \min_{t\in[\theta,1-\theta]}u(t)\geq \theta^{3}(1-\alpha+\beta) \|u\|,\]
where $\beta=\int_{\theta}^{1-\theta}a(t)dt$,\ \ $\alpha=\int_{0}^{1}a(t)dt$.
\end{lemma}
\begin{proof}
From Lemma \ref{lem 2.4} and Lemma \ref{lem 2.5}, $u(t)$ is nonnegative. For $t\in[0, 1]$, from Lemma \ref{lem 2.4} and Lemma \ref{lem 2.5}, we have that
\begin{eqnarray*}
u(t)&=&\int_{0}^{1}\Big(G(t, s)+\frac{1}{1-\alpha}\int_{0}^{1}a(\tau)G(\tau, s)d\tau \Big)y(s)ds\\
&\leq&\frac{1}{6}\int_{0}^{1}s(1-s)^{2}\Big(1+\frac{\alpha}{1-\alpha}\Big)y(s)ds\\
&=&\frac{1}{6(1-\alpha)}\int_{0}^{1}s(1-s)^{2}y(s)ds.
\end{eqnarray*}
Then
\begin{equation}\label{eq-2.10}
\|u\|\leq\frac{1}{6(1-\alpha)}\int_{0}^{1}s(1-s)^{2}y(s)ds,
\end{equation}
and for $t\in[\theta,1-\theta]$, we have
\begin{gather}
\begin{aligned}
u(t)&=\int_{0}^{1}\Big(G(t, s)+\frac{1}{1-\alpha}\int_{0}^{1}a(\tau)G(\tau, s)d\tau\Big)y(s)ds\\
&\geq\frac{\theta^{3}}{6}\int_{0}^{1}\Big[s(1-s)^{2}+\frac{1}{1-\alpha}\int_{\theta}^{1-\theta}s(1-s)^{2}a(\tau)d\tau\Big]y(s)ds\\
&=\frac{\theta^{3}}{6}\int_{0}^{1}s(1-s)^{2}\Big(1+\frac{\beta}{1-\alpha}\Big)y(s)ds\\
&=\frac{\theta^{3}}{6}\cdot\frac{1-\alpha+\beta}{1-\alpha}\int_{0}^{1}s(1-s)^{2}y(s)ds.
\end{aligned}\label{eq-2.11}
\end{gather}

From \eqref{eq-2.10}, \eqref{eq-2.11}, we obtain
\[ \min_{t\in[\theta,1-\theta]}u(t)\geq \theta^{3}(1-\alpha+\beta) \|u\|.\]
\end{proof}

Let $\theta\in]0, \frac{1}{2}[$. We define the cone
\[K=\left\{u\in C([0, 1], \mathbb{R}), u\geq0:  \min_{t\in[\theta,1-\theta]}u(t)\geq \theta^{3}(1-\alpha+\beta) \|u\|\right\},\]
and the operator $A:K\rightarrow C[0, 1]$ by

\begin{equation}\label{eq-2.12}
Au(t)=\int_{0}^{1}\Big(G(t, s)+\frac{1}{1-\alpha}\int_{0}^{1}a(\tau)G(\tau, s)d\tau\Big)f(u(s))ds.
\end{equation}

\begin{remark}
By Lemma \ref{lem 2.4}, problem \eqref{eq-1.1}, \eqref{eq-1.2} has a positive solution $u(t)$ if and only if $u$ is a fixed point of $A$.
\end{remark}

\begin{lemma}\label{lem 2.7}
The operator $A$ defined in \eqref{eq-2.12} is completely continuous and satisfies $AK\subset K$.
\end{lemma}
\begin{proof}
From Lemma \ref{lem 2.6}, we obtain $AK\subset K$.  $A$ is completely continuous by an application of Arzela-Ascoli theorem.
\end{proof}

In what follows, we will use the following notations\\
\begin{equation*}
f_{0}=\lim_{u\to 0+}\frac{f(u)}{u}, \ \
f_{\infty}=\lim_{u\to\infty}\frac{f(u)}{u}.
\end{equation*}
We note that the case $f_{0}=0$ and $f_{\infty}=\infty$ corresponds to the superlinear case and
$f_{0}=\infty$ and $f_{\infty}=0$ corresponds to the sublinear case.

\section{Existence of positive solutions}
In this section, we will state and prove our main results.

\begin{theorem}\label{theo 3.1}
Assume that $f_{0}=0$ and $f_{\infty}=\infty$. Then BVP \eqref{eq-1.1} and \eqref{eq-1.2}
has at least one positive solution.
\end{theorem}

\begin{proof}

Since $f_{0}=0$, there exists $\rho_{1}>0$ such that  $f(u)\leq \epsilon u$, for  $0<u\leq \rho_{1}$, where $\epsilon>0$ satisfies
\begin{equation*}
\frac{\epsilon}{6(1-\alpha)}\leq1.
\end{equation*}
Thus, if we let
\begin{equation*}
\Omega_{1}=\left\{u\in C[0,1]: \|u\|<\rho_{1}\right\},
\end{equation*}
then, for $u\in K\cap \partial\Omega_{1}$, we have
\begin{gather}
\begin{aligned}
Au(t)&=\int_{0}^{1}\Big(G(t, s)+\frac{1}{1-\alpha}\int_{0}^{1}a(\tau)G(\tau, s)d\tau\Big)f(u(s))ds\\
&\leq\frac{1}{6}\int_{0}^{1}\Big(s(1-s)^{2}+\frac{1}{1-\alpha}\int_{0}^{1}s(1-s)^{2}a(\tau)d\tau\Big)\epsilon u(s)ds\\
&\leq\frac{1}{6}\cdot\frac{\epsilon}{1-\alpha}\|u\|\int_{0}^{1}s(1-s)^{2}ds\\
&\leq\frac{1}{6}\cdot\frac{\epsilon}{1-\alpha}\|u\|\\
&\leq\|u\|.
\end{aligned}\label{eq-3.1}
\end{gather}

Therefore
\[\|Au\|\leq\|u\|, \ \ u\in K\cap\partial\Omega_{1}.\]

Further, since $f_{\infty}=\infty$, there exists $\widehat{\rho}_{2}>0$  such that $f(u)\geq \delta u$, for  $u> \widehat{\rho}_{2}$, where $ \delta>0$ is chosen so that

\begin{equation*}
\delta\frac{\theta^{6}}{36}\cdot\frac{(1-\alpha+\beta)^{2}}{1-\alpha}(1-2\theta)(\frac{1}{2}+\theta-\theta^{2})\geq1.
\end{equation*}

Let $\rho_{2}=\max\Big\{2\rho_{1}, \frac{\widehat{\rho}_{2}}{\theta^{3}(1-\alpha+\beta)}\Big\}$ and $\Omega_{2}=\left\{u\in C[0, 1]: \|u\|<\rho_{2}\right\}$. Then $u\in K\cap \partial\Omega_{2}$ implies that

\[ \min_{t\in[\theta,1-\theta]}u(t)\geq \theta^{3}(1-\alpha+\beta) \|u\|=\theta^{3}(1-\alpha+\beta)\rho_{2}\geq\widehat{\rho}_{2},\]
so, by \eqref{eq-2.12} and for $t\in[\theta, 1-\theta]$, we obtain

\begin{gather}
\begin{aligned}
Au(t)&=\int_{0}^{1}\Big(G(t, s)+\frac{1}{1-\alpha}\int_{0}^{1}a(\tau)G(\tau, s)d\tau\Big)f(u(s))ds\\
&\geq\int_{\theta}^{1-\theta}\Big[\frac{\theta^{3}}{6}s(1-s)^{2}+
\frac{1}{1-\alpha}\int_{\theta}^{1-\theta}\frac{\theta^{3}}{6}s(1-s)^{2}a(\tau)d\tau\Big]\delta u(s)ds\\
&=\frac{\theta^{3}}{6}\delta\int_{\theta}^{1-\theta}s(1-s)^{2}\Big(1+\frac{\beta}{1-\alpha}\Big)u(s)ds\\
&=\frac{\theta^{3}\delta}{6}\cdot\frac{(1-\alpha+\beta)}{1-\alpha}\int_{\theta}^{1-\theta}s(1-s)^{2}u(s)ds\\
&\geq\frac{\theta^{3}\delta}{6}\cdot\frac{(1-\alpha+\beta)}{1-\alpha}\min_{t\in[\theta,1-\theta]}u(t)\int_{\theta}^{1-\theta}s(1-s)^{2}ds\\
&\geq\delta\frac{\theta^{6}}{36}\cdot\frac{(1-\alpha+\beta)^{2}}{1-\alpha}(1-2\theta)(\frac{1}{2}+\theta-\theta^{2})\|u\|\\
&\geq\|u\|.
\end{aligned}\label{eq-3.2}
\end{gather}

Hence, $\|Au\|\geq\|u\|, \ u\in K\cap\partial\Omega_{2}$. By Theorem \ref{thm 2.3}, the operator $A$ has a fixed point in $K\cap (\overline{\Omega _{2}}$ $\backslash $ $
\Omega_{1})$ such that $\rho_{1}\leq \|u\|\leq \rho_{2}$.
\end{proof}

\begin{theorem}\label{theo 3.2}
Assume that $f_{0}=\infty$ and $f_{\infty}=0$. Then BVP \eqref{eq-1.1} and \eqref{eq-1.2}
has at least one positive solution.
\end{theorem}

\begin{proof}
Since $f_{0}=\infty$, there exists $\rho_{1}>0$ such that  $f(u)\geq \lambda u$, for  $0<u\leq \rho_{1}$, where $ \lambda>0$ is chosen so that
\begin{equation*}
\lambda\frac{\theta^{6}}{36}\cdot\frac{(1-\alpha+\beta)^{2}}{1-\alpha}(1-2\theta)(\frac{1}{2}+\theta-\theta^{2})\geq1.
\end{equation*}
Thus, for $u\in K\cap \partial\Omega_{1}$ with
\[\Omega_{1}=\left\{u\in C[0, 1]: \|u\|<\rho_{1}\right\},\]
we have from \eqref{eq-3.2}
\begin{eqnarray*}
Au(t)&=&\int_{0}^{1}\Big(G(t, s)+\frac{1}{1-\alpha}\int_{0}^{1}a(\tau)G(\tau, s)d\tau\Big)f(u(s))ds\\
&\geq&\int_{\theta}^{1-\theta}\Big[\frac{\theta^{3}}{6}s(1-s)^{2}+
\frac{1}{1-\alpha}\int_{\theta}^{1-\theta}\frac{\theta^{3}}{6}s(1-s)^{2}a(\tau)d\tau\Big]\lambda u(s)ds\\
&\geq&\lambda\frac{\theta^{6}}{36}\cdot\frac{(1-\alpha+\beta)^{2}}{1-\alpha}(1-2\theta)(\frac{1}{2}+\theta-\theta^{2})\|u\|\\
&\geq&\|u\|.
\end{eqnarray*}
Then, $ Au(t)\geq \|u\|$ for $ t\in[\theta, 1-\theta]$ , which implies that
\[ \|Au\|\geq\|u\|, \ \ u\in K\cap\partial\Omega_{1}.\]

Next we construct the set $\Omega_{2}$. We discuss two possible cases:\

Case 1. If $f$ is bounded. Then, there exists $L>0$ such that $f(u)\leq L$.
Set $\Omega_{2}=\left\{u\in C[0, 1]: \|u\|<\rho_{2}\right\}$, where $\rho_{2}=\max\Big\{2\rho_{1}, \frac{L}{6(1-\alpha)}\Big\}$.

If $u\in K\cap\partial\Omega_{2}$, similar to the estimates
of \eqref{eq-3.1}, we obtain
\begin{eqnarray}
Au(t)&\leq&\frac{1}{6}\cdot\frac{L}{1-\alpha}\int_{0}^{1}s(1-s)^{2}ds\\
&\leq&\frac{1}{6}\cdot\frac{L}{1-\alpha}\leq \rho_{2}=\|u\|,
\end{eqnarray}
and consequently, $\|Au\|\leq\|u\|, \ \ u\in K\cap\partial\Omega_{2}$.

Case 2. Suppose that $f$ is unbounded, since $f_{\infty}=0$, there exists  $\widehat{\rho}_{2}>0$ ($\widehat{\rho}_{2}>\rho_{1}$) such that $f(u)\leq \eta u$ for $u>\widehat{\rho}_{2}$, where $\eta>0$ satisfies
\[\frac{\eta}{6(1-\alpha)}\leq 1.\]
On the other hand, from condition {\rm (H1)}, there is $\sigma>0$ such that
$f(u)\leq \eta\sigma$,\ \ with $0\leq u\leq \widehat{\rho}_{2}$.

Now, set $\Omega_{2}=\left\{u\in C[0, 1]: \|u\|<\rho_{2}\right\}$, where  $\rho_{2}=\max\{\sigma,\widehat{\rho}_{2}\}$.

If $u\in K\cap\partial\Omega_{2}$, then we have $f(u)\leq\eta\rho_{2}$. Similar to \eqref{eq-3.1}, we obtain
\begin{eqnarray}
Au(t)&\leq&\frac{1}{6}\cdot\frac{\eta \rho_{2}}{1-\alpha}\int_{0}^{1}s(1-s)^{2}ds\\
&\leq&\frac{1}{6}\cdot\frac{\eta \rho_{2}}{1-\alpha}\leq \rho_{2}=\|u\|,
\end{eqnarray}
so, $\|Au\|\leq\|u\|$, for $u\in K\cap\partial\Omega_{2}$.
Therefore by Theorem \ref{thm 2.3}, $A$ has at least one fixed point, which is a positive solution of \eqref{eq-1.1} and \eqref{eq-1.2}.
\end{proof}

\section{Examples}

\begin{exmp}
Consider the boundary value problem

\begin{equation}\label{eq-4.1}
{u^{\prime \prime \prime \prime }}(t)+u^{2}(e^{-u}+1)=0, \  \ 0<t<1,
\end{equation}
\begin{equation}\label{eq-4.2}
u^{\prime}(0)=u^{\prime}(1)=u^{\prime \prime}(0)=0, \ u(0)=\int_{0}^{1}s^{2}u(s)ds,
\end{equation}
where $f(u)=u^{2}(e^{-u}+1)\in C([0,\infty),[0,\infty))$ and $a(t)=t^{2}\geq0$, $\int_{0}^{1}a(s)ds=\int_{0}^{1}s^{2}ds=\frac{1}{3}$.

We have
\begin{align*}
\lim_{u\to 0+}\frac{f(u)}{u}&=\lim_{u\to 0+}\frac{u^{2}(e^{-u}+1)}{u}=0,\\
\lim_{u\to+\infty}\frac{f(u)}{u}&=\lim_{u\to+\infty}\frac{u^{2}(e^{-u}+1)}{u}=+\infty.
\end{align*}

From Theorem \ref{theo 3.1}, the problem \eqref{eq-4.1} and \eqref{eq-4.2} has at least one positive solution.
\end{exmp}

\begin{exmp}
Consider the boundary value problem
\begin{equation}\label{eq-4.3}
{u^{\prime \prime \prime \prime }}(t)+\sqrt{1+u}+\sin{u}=0, \  \ 0<t<1,
\end{equation}
\begin{equation}\label{eq-4.4}
u^{\prime}(0)=u^{\prime}(1)=u^{\prime \prime}(0)=0, \ u(0)=\int_{0}^{1}su(s)ds,
\end{equation}
where $f(u)=\sqrt{1+u}+\sin{u}\in C([0,\infty),[0,\infty))$ and $a(t)=t\geq0$, $\int_{0}^{1}a(s)ds=\int_{0}^{1}sds=\frac{1}{2}$.

We have
\begin{align*}
\lim_{u\to 0+}\frac{f(u)}{u}&=\lim_{u\to 0+}\frac{\sqrt{1+u}+\sin{u}}{u}=+\infty,\\
\lim_{u\to+\infty}\frac{f(u)}{u}&=\lim_{u\to+\infty}\frac{\sqrt{1+u}+\sin{u}}{u}=0.
\end{align*}

Therefore, by Theorem \ref{theo 3.2}, the problem \eqref{eq-4.3} and \eqref{eq-4.4} has at least one positive solution.
\end{exmp}



\end{document}